\documentclass[10pt]{article}
\usepackage{latexsym}
\usepackage{amssymb}
\usepackage{amsmath}
\usepackage{amsthm}
\usepackage{thmtools}
\usepackage{cite}
\usepackage[all]{xy}
\usepackage{framed}
\usepackage{hyperref}
\usepackage{graphicx}
\usepackage{parskip}
\usepackage{hyperref}

\begingroup
    \makeatletter
    \@for\theoremstyle:=definition,remark,plain\do{%
        \expandafter\g@addto@macro\csname th@\theoremstyle\endcsname{%
            \addtolength\thm@preskip\parskip
            }%
        }
\endgroup

\makeatletter
\xydef@\txt@ii#1{\vbox{\vspace*{-5pt}%
 \let\\=\cr
 \tabskip=\z@skip \halign{\relax\hfil\txtline@@{##}\hfil\cr\leavevmode#1\crcr}}}
\makeatother

\theoremstyle{definition}
\newtheorem{thm}{Theorem}[section]
\newtheorem{lem}[thm]{Lemma}
\newtheorem{cor}[thm]{Corollary}
\newtheorem{defn}[thm]{Definition}
\newtheorem{propn}[thm]{Proposition}
\newtheorem*{thm*}{Theorem}

\newtheorem*{nts}{Note to self}

\theoremstyle{remark}
\newtheorem*{rk}{Remark}

\newtheorem{exs}[thm]{Examples}

\newtheoremstyle{custthm}{\parskip}{}{\normalfont}{}{\bfseries}{.}{ }{\thmname{#1} \thmnote{#3}}
\theoremstyle{custthm}

\newcommand{\nio}{\mathrm{nio}}

\newcommand{\tensor}[1]{\underset{#1}{\otimes}}

\newcommand{\gr}{\mathrm{gr}}

\newcommand{\fn}{\mathbf{FN}_p}

\newcommand{\cpi}{\mathrm{cpi}^{\overline{k\Delta^+}}}
\newcommand{\cpisum}[2]{#1|^{#2}}

\newcommand{\Kdim}{\mathrm{Kdim}}

{%
\begin{framed}\begin{nts}%
}%
{%
\end{nts}\end{framed}%
}

\begin{document}

\binoppenalty=\maxdimen
\relpenalty=\maxdimen

\title{On the catenarity of virtually nilpotent mod-$p$ Iwasawa algebras}
\author{William Woods}
\date{\today}
\maketitle
\begin{abstract}
Let $p>2$ be a prime, $k$ a finite field of characteristic $p$, and $G$ a nilpotent-by-finite compact $p$-adic analytic group. Write $kG$ for the completed group ring of $G$ over $k$. We show that $kG$ is a catenary ring.
\end{abstract}

\newpage
\tableofcontents

%%
%
%\renewcommand{\listtheoremname}{List of theorems}
%\listoftheorems[ignoreall,show=thm]
%
%\hyperref[thm: theorem A]{Theorem A}
%
%\renewcommand{\listtheoremname}{List of definitions}
%\listoftheorems[ignoreall,show=defn]
%
%\renewcommand{\listtheoremname}{List of lemmas}
%\listoftheorems[ignoreall,show=lem]
%
%\renewcommand{\listtheoremname}{List of corollaries}
%\listoftheorems[ignoreall,show=cor]
%
%\renewcommand{\listtheoremname}{List of propositions}
%\listoftheorems[ignoreall,show=propn]

%
\newpage
\section*{Introduction}
\addcontentsline{toc}{section}{Introduction}

Fix a prime $p$, a commutative pseudocompact ring $k$ (e.g. $\mathbb{F}_p$ or $\mathbb{Z}_p$) and a compact $p$-adic analytic group $G$. (Such groups are perhaps most accessibly characterised as those groups $G$ which are isomorphic to a closed subgroup of $GL_n(\mathbb{Z}_p)$ for some $n$.) The \emph{completed group ring} $kG$ (sometimes written $k[[G]]$) is defined by
$$kG := \underset{N}{\underleftarrow{\lim}}\; k[G/N],$$
where the inverse limit ranges over all open normal subgroups $N$ of $G$, and $k[G/N]$ denotes the ordinary group algebra of the (finite) group $G/N$ over $k$. This ring satisfies an obvious universal property \cite[Lemma 2.2]{woods-prime-quotients}, and modules over it characterise \emph{continuous} $k$-representations of $G$ (which has the profinite topology). When $k = \mathbb{F}_p$, $\mathbb{Z}_p$ or related rings, this is often called the \emph{Iwasawa algebra} of $G$.

Iwasawa algebras (and related objects, such as locally analytic distribution algebras \cite{schn-teit1}) have recently become a very active research area due to their number-theoretic interest, for instance in the $p$-adic Langlands programme: see \cite{schn-teit2}, for example. They are also interesting objects of study in their own right, as an interesting class of noetherian rings: see \cite{ardakovbrown} for a 2006 survey of what is known about these rings.

Our main result is the following.

\textbf{Theorem A.} Take $p>2$. Let $G$ be a nilpotent-by-finite compact $p$-adic analytic group, and let $k$ be a finite field of characteristic $p$. Then $kG$ is a catenary ring.\qed

Recall that a ring $R$ is said to be \emph{catenary} if any two maximal chains of prime ideals with common endpoints have the same length, i.e. whenever
$$P = P_1 \lneq P_2 \lneq \dots \lneq P_r = P'$$
$$P = Q_1 \lneq Q_2 \lneq \dots \lneq Q_s = P'$$
are two chains of prime ideals of $R$ which cannot be refined further (i.e. by adding an extra prime ideal $P_i \lneq I \lneq P_{i+1}$ or $Q_i \lneq I \lneq Q_{i+1}$), we have that $r=s$. This is a ``well-behavedness" condition on the classical Krull dimension of $kG$: it says that, whenever $P\lneq P'$ are \emph{adjacent} prime ideals and the height $h(P)$ of $P$ is finite, then we have $h(P') = h(P) + 1$.

This result goes some way towards redressing the long-standing gap between Iwasawa algebras and similar algebraic objects; for instance, similar catenarity results had already been established for classical group rings of virtually polycyclic groups (in a special case in \cite{roseblade}, in full generality in \cite{ll}), for universal enveloping algebras of finite-dimensional soluble Lie algebras over $\mathbb{C}$ \cite{gabber}; for quantised coordinate rings over $\mathbb{C}$ \cite{goodearl-zhang}, and over \cite{yakimov}; for $q$-commutative power series rings \cite{letzter-wang}; and so on.

In proving this result, we crucially use the prime extension theorem, \cite[Theorem A]{woods-extensions-of-primes}. Before we can state this, we need to recall a few concepts.

Let $G$ be a nilpotent-by-finite compact $p$-adic analytic group. Then \cite[Theorem C]{woods-struct-of-G} there exists a unique maximal subgroup $H$ of $G$ with the property that $H$ is an open normal subgroup of $G$, $H$ contains a finite normal subgroup $F$, and $H/F$ is nilpotent $p$-valuable. This $H$ is called the \emph{finite-by-(nilpotent $p$-valuable) radical} of $G$, and is written $\fn(G)$.

Given a prime ideal $P$ of $kG$, as in \cite[\S 1.3]{ardakovInv} and \cite[Introduction]{woods-extensions-of-primes}, we will write $P^\dagger$ for the kernel of the natural composite map $G\to (kG)^\times\to (kG/P)^\times$. We say that $P$ is \emph{faithful} if $P^\dagger = 1$, and $P$ is \emph{almost faithful} if $P^\dagger$ is finite.

We now state the prime extension theorem:

\begin{thm*}\cite[Theorem A]{woods-extensions-of-primes} 
Take $p>2$. Let $G$ be a nilpotent-by-finite compact $p$-adic analytic group, and let $k$ be a finite field. Write $H = \fn(G)$. If $P$ is an almost faithful prime ideal of $kH$, then the ideal $PkG$ is a prime ideal of $kG$.\qed
\end{thm*}

We will use this result to generalise Ardakov's analogue of Zalesskii's theorem \cite[Theorem 8.6]{ardakovInv} to our current context.

Recall \cite[Definition 1.4, Lemma 1.10]{woods-struct-of-G} that a group $G$ is \emph{orbitally sound} if, whenever $H$ is a subgroup of $G$ with finitely many $G$-conjugates, and $H^\circ$ is the largest normal subgroup of $G$ contained in $H$, we have $[H:H^\circ]<\infty$; and recall that nilpotent $p$-valuable groups are indeed orbitally sound \cite[Proposition 5.9]{ardakovInv}.

As in \cite{woods-struct-of-G}, we will write throughout this paper
\begin{align*}
\Delta(G)^{\hphantom{+}} &= \big\{x\in G \;\big|\; [G: \mathbf{C}_G(x)] < \infty\big\},\\
\Delta^+(G) &= \big\{x\in \Delta \;\big|\; o(x) < \infty\big\},
\end{align*}
where $o(x)$ denotes the order of $x$. We will also often simply write $\Delta$ and $\Delta^+$ to denote $\Delta(G)$ and $\Delta^+(G)$. For the basic properties of these (closed, characteristic) subgroups, see \cite[Lemma 1.3 and Theorem D]{woods-struct-of-G}.

Following Roseblade \cite{roseblade}, we say that a prime ideal $P$ of $kG$ is \emph{controlled} by the normal subgroup $H$ of $G$ if the right ideal $(P\cap kH)kG$ is equal to $P$.

\textbf{Theorem B.} Let $G$ be a nilpotent-by-finite, orbitally sound compact $p$-adic analytic group. Suppose $P$ is an almost faithful prime ideal of $kG$. Then $P$ is controlled by $\Delta$.\qed

The analogous classical result, for group algebras of polycyclic-by-finite groups, was proved by Roseblade \cite[Corollary H3]{roseblade}. Taken together with the results of \cite{ardakovGMJ} and \cite{woods-prime-quotients}, this also gives a precise partial answer to a question of Ardakov and Brown \cite[Question G]{ardakovbrown}: when $G$ is as in Theorem B, this completely describes the prime ideals of $kG$ in terms of closed normal subgroups, central elements and prime ideals of (classical) group algebras of \emph{finite} groups over $k$.

This is all we need to deduce that $kG$ is catenary when $G$ is nilpotent-by-finite and orbitally sound - see Theorem \ref{thm: catenarity in o.s. case}. In order to pass from orbitally sound to general nilpotent-by-finite groups, we partly develop the theory of \emph{vertices and sources} along the lines of \cite{lorenz}\cite{lp}.

\textbf{Theorem C.} Let $G$ be a nilpotent-by-finite compact $p$-adic analytic group, $P$ a prime ideal of $kG$, $H$ an orbitally sound open normal subgroup of $G$, and $Q$ a minimal prime ideal above $P\cap kH$. Write $N$ for the $G$-isolator \cite[Definition 1.6]{woods-struct-of-G} of $Q^\dagger$, and write $\nabla$ for the subgroup of $G$ containing $N$ defined by $\nabla/N = \Delta(G/N)$. Then $P$ is \emph{induced} from an ideal $L$ of $k\nabla$.\qed

For the precise meaning of \emph{induced} here, see \S 1.2.

Theorem A then follows from Theorems B and C by adapting an argument from \cite{ll}, as follows. Let $G$ be a (not necessarily orbitally sound) nilpotent-by-finite compact $p$-adic analytic group, $k$ a finite field of characteristic $p>2$, and $P$ a faithful prime ideal of $kG$. We already know \cite[Theorem A]{woods-struct-of-G} that $G$ contains an open normal orbitally sound subgroup, which we denote $\nio(G)$. From Theorem C, we may deduce Corollary \ref{cor: faithful primes are induced from a proper open subgroup}: that $P$ is induced from some proper open subgroup $H$ of $G$ containing $\nio(G)$. If $H = \nio(G)$, then we can deduce from Theorem B (as above) that $kH$ is catenary, and now by Lemma \ref{lem: catenary subgroup} we are done. In general, $\nio(G) \leq H < G$, and we may not have equality: but it is easy to see that $[H:\nio(H)] < [G:\nio(G)] < \infty$, and Corollary \ref{cor: main result} establishes Theorem A by induction on the index $[G:\nio(G)]$.

\newpage
\section{Heights of primes and Krull dimension}\label{section: heights}

\subsection{Prime and $G$-prime ideals}

\begin{defn}\label{defn: G-prime}
Let $G$ be a compact $p$-adic analytic group \cite[Definition 8.14]{DDMS}. Suppose the group $G$ acts (continuously) on the ring $R$, and that the ideal $I\lhd R$ is $G$-stable. Then, following \cite[\S 14]{passmanICP}, we will say that $I$ is \emph{$G$-prime} if, whenever $A, B\lhd R$ are $G$-stable ideals and $AB\subseteq I$, then either $A\subseteq I$ or $B\subseteq I$.
\end{defn}

\begin{lem}\label{lem: G-prime stuff}
Let $G$ be a compact $p$-adic analytic group and $H$ a closed normal subgroup.
\begin{itemize}
\item[(i)] If $P$ is a prime ideal of $kG$, then $P\cap kH$ is a $G$-prime ideal of $kH$. If $H$ is open in $G$, then $P$ is a minimal prime ideal above $(P\cap kH)kG$.
\item[(ii)] Let $Q$ be a $G$-prime ideal of $kH$, and $P$ any minimal prime of $kH$ above $Q$. Then $Q = \bigcap_{x\in G} P^x$. Furthermore, the set of minimal primes of $kG$ above $Q$ is $\{P^x | x\in G\}$.
\end{itemize}
\end{lem}

\begin{proof}
$ $

\begin{itemize}
\item[(i)] The former statement follows from \cite[Lemma 14.1(i)]{passmanICP}, and the latter from \cite[Theorem 16.2(i)]{passmanICP}.
\item[(ii)] This follows from \cite[Lemma 14.2(i)(ii)]{passmanICP}.\qedhere
\end{itemize}
\end{proof}

\begin{defn}\label{defn: G-height}
Let $P$ be a prime ideal of a ring $R$. Then we define the \emph{height} of $P$ to be the greatest integer $h(P) := r$ for which there exists a (finite) chain
\begin{equation}
P_0 \lneq P_1 \lneq \dots \lneq P_r = P\tag{$\dagger$}
\end{equation}
of prime ideals in $R$ (or $\infty$ if no such longest finite chain exists).

Suppose instead that the group $G$ acts on $R$ by automorphisms, and $P$ is a $G$-prime ideal of $R$. Then the \emph{$G$-height} of $P$ is the greatest integer $h_G(P) := r$ for which there exists a chain ($\dagger$) of $G$-prime ideals in $R$ (or $\infty$).

Finally, suppose that the group $G$ acts on $R$ by automorphisms, and $P$ is a $G$-orbital prime ideal of $R$ (i.e. a prime ideal of $R$ with finite orbit under the conjugation action of $G$). Then the \emph{$G$-orbital height} of $P$ is the greatest integer $h_G^{\mathrm{orb}}(P) := r$ for which there exists a chain ($\dagger$) of $G$-orbital prime ideals in $R$ (or $\infty$).
\end{defn}

We note the following immediate consequence of the correspondence of Lemma \ref{lem: G-prime stuff}:

\begin{cor}\label{cor: height and G-height of nearby things}
Let $G$ be a compact $p$-adic analytic group and $H$ an open normal subgroup. Take $P$ a prime ideal of $kG$, and let $Q$ be a minimal prime of $kH$ above $P\cap kH$. Then $h(P) = h_G(P\cap kH) = h(Q)$.\qed
\end{cor}

\subsection{Inducing ideals}

\begin{defn}\label{defn: induced ideals}
Let $H$ be an open (not necessarily normal) subgroup of $G$, and let $L$ be an ideal of $kH$. We define the \emph{induced ideal} $L^G\lhd kG$ to be the largest (two-sided) ideal contained in the right ideal $LkG\underset{r}{\lhd} kG$. In other words, by \cite[2.1]{ll}, $L^G$ is the annihilator of $kG/LkG$ as a right $kG$-module, or by \cite[Lemma 14.4(ii)]{passmanICP},
$$L^G = \bigcap_{g\in G} L^g kG.$$
\end{defn}

\begin{lem}\label{lem: induction of ideals}
Induction of ideals is \emph{transitive}: if $H$ and $K$ are open subgroups of $G$ with $H \leq K \leq G$, and $L\lhd kH$, then $L^G = (L^K)^G$.
\end{lem}

\begin{proof}
Let $N$ be an open normal subgroup of $G$ contained in $H$, and write $\overline{(\cdot)}$ to denote the quotient by $N$, so that we have $kG = kN * \overline{G}$ with $\overline{H}\leq \overline{K}\leq \overline{G}$, and we may view $L$ as an ideal of $kN*\overline{H}$. The result now follows from \cite[Lemma 1.2(iii)]{lp}.
\end{proof}

\subsection{Krull dimension}

We recall some facts about Krull dimension, used here in the sense of Gabriel and Rentschler: see \cite[\S 15]{GW}.

\begin{defn}
Let $0\neq M$ be an $R$-module, and fix some ordinal $\alpha$. We define the following notation inductively:
\begin{itemize}
\item $\Kdim(M) = 0$ if $M$ is an Artinian module,
\item $\Kdim(M) \leq \alpha$ if, for every descending chain $$M_0 \geq M_1 \geq M_2 \geq \dots$$ of submodules of $M$, we have $\Kdim(M_i/M_{i+1}) < \alpha$ for all but finitely many $i$.
\end{itemize}
Of course, if there exists some $\alpha$ such that $\Kdim(M)\leq \alpha$, but we do not have $\Kdim(M) \leq \beta$ for any $\beta < \alpha$, then we write $\Kdim(M) = \alpha$.
\end{defn}

\begin{rk}
$\Kdim(M)$ is a measure of complexity of the poset of submodules of $M$.

$\Kdim(M)$ may not be defined for some modules $M$ -- that is, we may not have $\Kdim(M) \leq \alpha$ for any ordinal $\alpha$. However, if $M$ is a noetherian module, then $\Kdim(M)$ is defined \cite[Lemma 15.3]{GW}.
\end{rk}

\begin{defn}
Suppose that $\Kdim(M) = \alpha$. We say that $M$ is \emph{$\alpha$-homogeneous} if $\Kdim(N) = \alpha$ for all nonzero submodules $N$ of $M$.
\end{defn}

\begin{exs}\label{exs: alpha-hgs rings}
$ $

\begin{itemize}
\item[(i)] Nonzero Artinian modules are $0$-homogeneous.
\item[(ii)] Prime rings $R$, as modules over themselves, are $\alpha$-homogeneous (where we set $\alpha$ equal to $\Kdim(R_R)$) \cite[Exercise 15E]{GW}.
\item[(iii)] The property of being $\alpha$-homogeneous is inherited by products \cite[Corollary 15.2]{GW} and (nonzero) submodules (by definition).
\end{itemize}
\end{exs}

We now cite and adapt some standard results on Krull dimension.

\begin{lem}\label{lem: Kdim results}
$ $

\begin{itemize}
\item[(i)] \hspace{-1pt}\cite[1.4(ii)]{ll}
Let the ring $R$ be $\alpha$-homogeneous as a right $R$-module. If $x\in R$ satisfies $\Kdim(R/xR) < \Kdim(R)$, then $x$ is a regular element of $R$.
\item[(ii)] \hspace{-1pt}\cite[Th\'eor\`eme 5.3]{lemonnier}
Suppose $B$ is a finite normalising extension of $A$, and let $M$ be a $B$-module. Then $\Kdim(M_B)$ exists if and only if $\Kdim(M_A)$ does, and if so, then they are equal.
\item[(iii)] \hspace{-1pt}\cite[Exercise 15R]{GW}
If $R$ is a right noetherian subring of a ring $S$ such that $S$ is finitely generated as an $R$-module, and $M$ is a finitely generated $S$-module, then $\Kdim(M_S)\leq \Kdim(M_R)$.\qed
\end{itemize}
\end{lem}

\begin{cor}\label{cor: Kdims are equal}
Suppose $A\subseteq C\subseteq B$ are right noetherian rings, and $B$ is a finite normalising extension of $A$. Let $M$ be a finitely generated $B$-module. Then, if $\Kdim(M_B)$ exists, we have
$$\Kdim(M_A) = \Kdim(M_C) = \Kdim(M_B).$$
\end{cor}

\begin{proof}
This follows immediately from Lemma \ref{lem: Kdim results}(ii) and two applications of Lemma \ref{lem: Kdim results}(iii).
\end{proof}

\begin{lem}\label{lem: Kdim and moving modules around}
Let $G$ be a compact $p$-adic analytic group, $H$ an open subgroup of $G$, and $k$ a field of characteristic $p$. Let $M$ be a finitely generated $kG$-module.
\begin{itemize}
\item[(i)] $\Kdim(M_{kG}) = \Kdim(M_{kH})$.
\item[(ii)] Suppose that $M = WkG$ for some submodule $W$ of $M_{kH}$. Then we have $\Kdim(M_{kG}) = \Kdim(W_{kH})$.
\item[(iii)] $M_{kG}$ is $\alpha$-homogeneous if and only if $M_{kH}$ is $\alpha$-homogeneous.
\end{itemize}
\end{lem}

\begin{proof}
(Adapted from \cite[1.4(iii)-(v)]{ll}.)

\begin{itemize}
\item[(i)] 
Let $N$ be the (open) largest normal subgroup of $G$ contained in $H$, so that $kG$ is a finite normalising extension of $kN$. Now apply Corollary \ref{cor: Kdims are equal}.
\item[(ii)] Let $N$ be as in (i). Then, by (i), it suffices to prove that $\Kdim(M_{kN}) = \Kdim(W_{kN})$. But, as a $kN$-module, $M$ is a finite sum of modules $(Wg)_{kN}$ for various $g\in G$, and these are all isomorphic, so in particular have isomorphic submodule lattices and therefore the same $\Kdim$.
\item[(iii)] It is clear from the definition that, if $M_{kH}$ is $\alpha$-homogeneous, then $M_{kG}$ is $\alpha$-homogeneous. Conversely, suppose that $M_{kG}$ is $\alpha$-homogeneous, and let $W$ be a nonzero submodule of $M_{kH}$. Then $(WkG)_{kG}$ is a nonzero submodule of $M_{kG}$, so has Krull dimension $\alpha$ by assumption, and hence also $\Kdim(W_{kH}) = \alpha$ by (ii).\qedhere
\end{itemize}
\end{proof}

\begin{lem}\label{lem: brown}
Let $G$ be a finite group, $H$ a subgroup, and $R*G$ a fixed crossed product. Fix a semiprime ideal $I$ of $R*G$. If $R*G/I$ is $\alpha$-homogeneous, then $R*H/(I\cap R*H)$ is $\alpha$-homogeneous.
\end{lem}

\begin{proof}
(Adapted from \cite[Lemma 4.2(i)]{brown}.)
Let $M$ be a nonzero right ideal of the ring $R*H/(I\cap R*H)$, and write $\beta = \Kdim(M_{R*H})$. We wish to show that $\beta = \alpha$.

$M$ is a right module over both $R*H$ and $R$; and $R*G/I$ is a right module over both $R*G$ and $R$. As $R*G$ and $R*H$ are both finite normalising extensions of $R$, we may apply Lemma \ref{lem: Kdim results}(ii) to both of these situations to see that
$$\beta = \Kdim(M_{R*H}) = \Kdim(M_{R})$$
and
$$\alpha = \Kdim((R*G/I)_{R*G}) = \Kdim((R*G/I)_{R}).$$
Now, as right $R$-modules, we have
$$R*H/(I\cap R*H) \cong (R*H + I)/I \leq R*G/I,$$
and so $M$ is isomorphic to some nonzero $R$-submodule of $R*G/I$. In particular, this means that
$$\beta = \Kdim(M_R) \leq \Kdim((R*G/I)_{R}) = \alpha.$$
But now $(R*G/I)_R$ is $\alpha$-homogeneous by Corollary \ref{cor: Kdims are equal}, so we must have $\beta = \alpha$.
\end{proof}

\begin{cor}\label{cor: related rings are all alpha-homogeneous}
Let $G$ be a compact $p$-adic analytic group, $H$ be an open subgroup of $G$, and $N$ the largest open normal subgroup of $G$ contained in $H$. Take $k$ to be a field of characteristic $p$, and let $Q$ be a prime ideal of $kH$, $I = Q^G\cap kN$, and $\alpha = \Kdim(kH/Q)$. Then $kH/Q$, $kG/Q^G$, $kG/IkG$ are all $\alpha$-homogeneous rings.
\end{cor}

\begin{proof}
As we observed in Example \ref{exs: alpha-hgs rings}(ii), $kH/Q$ is already $\alpha$-homogeneous, as it is prime of Krull dimension $\alpha$.

We know from Definition \ref{defn: induced ideals} that the ideal $Q^G$ can be written as $\bigcap_{g\in G} Q^gkG$, and that this intersection can be taken to be finite. Hence, as a right $kG$-module, $kG/Q^G$ is isomorphic to a (nonzero) submodule of the direct product of the various (finitely many) $kG/Q^gkG$; and each $kG/Q^gkG$ is generated as a $kG$-module by $kH^g/Q^g$, which is ring-isomorphic to $kH/Q$. Hence $\Kdim(kG/Q^G) = \Kdim(kH/Q)$ by Lemma \ref{lem: Kdim results}(ii).

Finally, as $Q^G = \bigcap_{g\in G} (QkG)^g$, we see that
$$I = \bigcap_{g\in G} (QkG)^g\cap kN = \bigcap_{g\in G} (QkG\cap kN)^g = \bigcap_{g\in G} (Q\cap kN)^g,$$
and so, as above, $kN/I$ is a (nonzero) subdirect product of the various $kN/(Q\cap kN)^g$, which are all ring-isomorphic to $kN/Q\cap kN$; now Lemma \ref{lem: brown} implies that $kN/Q\cap kN$ is $\alpha$-homogeneous, so $kN/I$ is also, and $kG/IkG$ is generated as a $kG$-module by $kN/I$, so finally $kG/IkG$ also inherits this property.
\end{proof}

We borrow a result from the standard proof of Goldie's theorem.

\begin{lem}\label{lem: swan}
\cite[Lemma 3.13]{swan}
Suppose $R$ is a semiprime ring, satisfying the ascending chain condition on right annihilators of elements, and which does not contain an infinite direct sum of nonzero right ideals. If $I$ is an \emph{essential} right ideal of $R$ (i.e. a right ideal that has nonzero intersection $I\cap J$ with each nonzero right ideal $J$ of $R$), then $I$ contains a regular element.\qed
\end{lem}

These hypotheses are satisfied when $R$ is $G$-prime and noetherian, for example.

\begin{propn}\label{propn: P minimal over Q implies heights are equal}
With notation as in Corollary \ref{cor: related rings are all alpha-homogeneous}, suppose $P$ is a prime ideal of $kG$ containing $Q^G$. If $P$ is minimal over $Q^G$, then $h(P) = h(Q)$.
\end{propn}

\begin{proof}
First, set $I = Q^G\cap kN$. This is a $G$-prime ideal contained in $P\cap kN$.

Suppose for contradiction that the inclusion $I\subseteq P\cap kN$ is strict.

First, we will show that $P\cap kN/I$ is essential as a right ideal inside $kN/I$. Indeed, the left annihilator $L$ in $kN/I$ of $P\cap kN/I$ is a $G$-invariant ideal which annihilates the nonzero $G$-invariant ideal $P\cap kN/I$, so we must have $L = 0$; and so, given any right ideal $T$ of $kN/I$ having zero intersection with $P\cap kN/I$, as we must have $T\leq L$, we conclude that $T = 0$.

Hence, by Lemma \ref{lem: swan}, we may find an element $c\in P\cap kN\subseteq kN$ which is regular modulo $I$. As $kG/IkG$ is a free $kN/I$-module, $c$ may also be considered as an element of $P\subseteq kG$ which is regular modulo $IkG$. Hence
\begin{align*}
&\Kdim\big( kG/(Q^G + ckG)\big)_{kG}\\
&\qquad\qquad \leq \Kdim\big( kG/(IkG + ckG)\big)_{kG}&\text{as } IkG + ckG\subseteq Q^G + ckG\\
&\qquad\qquad < \Kdim(kG/IkG)_{kG}&\text{by Lemma \ref{lem: Kdim results}(i)}\\
&\qquad\qquad = \Kdim(kG/Q^G)_{kG}&\text{by Corollary \ref{cor: related rings are all alpha-homogeneous},}
\end{align*}
which, again by Lemma \ref{lem: Kdim results}(i), shows that $c\in P$ is regular modulo $Q^G$.

However, we may now deduce from a reduced rank argument that $P$ cannot be minimal over $Q^G$, as follows. Write $\rho$ for the reduced rank \cite[\S 11, Definition]{GW} of a right module over the semiprime noetherian (hence Goldie) ring $R = kG/Q^G$, and write $\overline{(\cdot)}$ for images under the map $kG\to R$. Now, $c\in P$ implies $\overline{c} R \subseteq \overline{P}$, and so by \cite[Lemma 11.3]{GW} we have $\rho(R/\overline{c} R)\geq \rho(R/\overline{P}) \geq 0$. Further, if $\overline{c}$ is a regular element of $R$, then $\overline{c}R \cong R$ as right $R$-modules, so $\rho(R/\overline{c}R) = 0$, again by \cite[Lemma 11.3]{GW}. But now \cite[Exercise 11C]{GW} implies that $\overline{P}$ cannot be a minimal prime of $R$.

This contradicts the assumption we made at the start of the proof, and so we have shown that $P\cap kN = Q^G\cap kN$.

We observed during the proof of Corollary \ref{cor: related rings are all alpha-homogeneous} that $$Q^G\cap kN = \bigcap_{g\in G} (Q\cap kN)^g.$$ But $Q$ is a prime ideal of $kH$, so $Q\cap kN$ is an $H$-prime ideal of $kN$, so may be written as $$Q\cap kN = \bigcap_{h\in H} Q_0^h$$ for some prime ideal $Q_0$ of $kN$. Combining these two shows that $$P\cap kN = Q^G\cap kN = \bigcap_{g\in G} Q_0^g.$$
Now, by applying \cite[corollary 16.8]{passmanICP} to both $P\cap kN$ and $Q\cap kN$, we have that $$h(P) = h(Q_0) = h(Q)$$ as required.
\end{proof}

%%%%
\iftrue

\section{Control theorem}\label{section: control}

\subsection{The abelian case}

We will require some facts about prime ideals in power series rings.

\begin{lem}\label{lem: passing between abelian groups and isolated subgroups}
Let $A$ be a free abelian pro-$p$ group of finite rank and $B$ a closed isolated (normal) subgroup. Take $k$ to be a field of characteristic $p$. Write $\mathrm{Spec}^B (kA)$ for the set of primes of $kA$ that are controlled by $B$. Then the maps
\begin{align*}
\mathrm{Spec}^B(kA) &\leftrightarrow \mathrm{Spec}(kB)\\
P&\mapsto P\cap kB\\
QkA\,&\text{\reflectbox{$\mapsto$}} \;Q
\end{align*}
are well-defined and mutual inverses, and preserve faithfulness.
\end{lem}

\begin{proof}
If $P$ is a prime ideal of $kA$, then $P\cap kB$ is an $A$-prime ideal (and hence a prime ideal) of $kB$ by Lemma \ref{lem: G-prime stuff}(i).

Conversely, note that, as $B$ is isolated in $A$, the quotient $A/B$ is again free abelian pro-$p$; so we may write $A = B\oplus C$, where the natural quotient map $A\to A/B$ induces an isomorphism $A/B\cong C$. Now, if $Q$ is a prime ideal of $kB$, then $kA/QkA = (kB/Q)[[C]]$ is a power series ring with coefficients in the commutative domain $kB/Q$, and is hence itself a domain.

It follows from \cite[Lemma 5.1]{ardakovGMJ} that $QkA\cap kB = Q$, and by assumption, if $P$ is controlled by $B$ then we already have $(P\cap kB)kA = P$.

Now suppose the prime ideals $P\lhd kA$ and $Q\lhd kB$ correspond under these maps. Then, again viewing $A$ as $B\oplus C$, we may similarly consider $kA/P$ as the completed tensor product \cite[Definition 2.3]{woods-prime-quotients} $kB/Q \hat{\tensor{k}} kC$. Then the map $A\to (kA/P)^\times$ can be written as
\begin{align*}
B\oplus C &\to (kB/Q)^\times \oplus (kC)^\times \lesssim (kB/Q \hat{\tensor{k}} kC)^\times\\
(b, c) &\mapsto ((b+Q), c),
\end{align*}
so it is clear that $P$ is faithful if and only if $Q$ is faithful.
\end{proof}

\begin{lem}\label{lem: commutative prime facts}
Let $A$, $B$, $k$ be as in Lemma \ref{lem: passing between abelian groups and isolated subgroups}. Take two neighbouring prime ideals $P\lneq Q$ of $kA$, and suppose $B$ controls $P$. Then
\begin{itemize}
\item[(i)] $h(P) + \dim (A/P) = r(A)$,
\item[(ii)] $h(Q) = h(P) + 1$,
\item[(iii)] $h(P) = h(P\cap kB)$.
\end{itemize}
\end{lem}

\begin{proof}
$ $

\begin{itemize}
\item[(i)] This follows from \cite[Ch. VII, \S 10, Corollary 1]{zs}.
\item[(ii)] This follows from \cite[Ch. VII, \S 10, Corollary 2]{zs}.
\item[(iii)] Under the correspondence of Lemma \ref{lem: passing between abelian groups and isolated subgroups}, any saturated chain of prime ideals $0 = Q_0 \lneq Q_1 \lneq \dots \lneq Q_n = P\cap kB$ of $kB$ extends to a chain of prime ideals $0 = P_0 \lneq P_1 \lneq \dots \lneq P_n = P$ of $kA$. As any two saturated chains of prime ideals in $kA$ have the same length \cite[Ch. VII, \S 10, Theorem 34 and Corollary 1]{zs}, we need only check that this chain is saturated.

Take two adjacent prime ideals $I_1\lneq I_2$ of $kB$, so that $h(I_2) = h(I_1) + 1$ \cite[Ch. VII, \S 10, Corollary 2]{zs} and $I_1kA \lneq I_2kA$ are prime. We will show that $I_1kA$ and $I_2kA$ are adjacent by showing that their heights also differ by 1. By performing induction on $r(A/B)$, it will suffice to prove this for the case $r(A/B) = 1$, i.e. $kA = kB[[X]]$.

It is clear that, when $R$ is a commutative ring, $\dim(R[[X]]) \geq 1 + \dim(R)$ (where $\dim$ denotes the classical Krull dimension). But, giving $R[[X]]$ the $(X)$-adic filtration, we see that $\gr(R[[X]]) \cong R[x]$. By \cite[6.5.6]{MR}, we have $\dim(R[[X]]) \leq \dim(\gr(R[[X]])) = \dim(R[x]) = 1 + \dim(R)$, where this last equality follows from \cite[6.5.4(i)]{MR}.

Hence, for any prime ideal $I$, we have $$\dim(kA/IkA) - \dim(kB/I) = \dim((kB/I)[[X]]) - \dim(kB/I) = 1.$$ But, from (i), we see that
\begin{align*}
\dim(kA/IkA) &= r(A) - h(IkA),\\
\dim(kB/I) &= r(B) - h(I),
\end{align*} and hence we conclude that $h(I) = h(IkA)$. Setting $I = I_1, I_2$ now shows that $h(I_2kA) = h(I_1kA) + 1$ as required.
\qedhere
\end{itemize}
\end{proof}

\subsection{Faithful primes are controlled by $\Delta$}

As in \cite{woods-struct-of-G}, if $P$ is a prime ideal of some completed group ring $kG$, we will write $P^\dagger := \ker(G\to (kG/P)^\times)$, and say that $P$ is \emph{faithful} if $P^\dagger = 1$ and $P$ is \emph{almost faithful} if $P^\dagger$ is finite.

Fix a prime $p$, which will be arbitrary until otherwise stated.

Recall the control theorem of Ardakov \cite[8.6]{ardakovInv}:

\begin{thm}\label{thm: control, nilpt p-valued}
Let $G$ be a nilpotent $p$-valued group of finite rank with centre $Z$.
\begin{itemize}
\item[(i)] If $\mathfrak{p}$ is a prime ideal of $kZ$, then $\mathfrak{p}kG$ is a prime ideal of $kG$.
\item[(ii)] If $P$ is a faithful prime ideal of $kG$, then $P$ is controlled by $Z$.
\end{itemize}
\end{thm}

\begin{proof}
This is \cite[8.4, 8.6]{ardakovInv}.
\end{proof}

\begin{lem}\label{lem: centre = delta}
Let $G$ be finite-by-(nilpotent $p$-valuable), i.e. $G = \fn(G)$. Then $Z(G/\Delta^+) = \Delta/\Delta^+$.
\end{lem}

\begin{proof}
Given $x\in G$, the two conditions $[G/\Delta^+: \mathbf{C}_{G/\Delta^+}(x\Delta^+)] < \infty$ and $[G: \mathbf{C}_G(x)] < \infty$ are equivalent, as $\Delta^+$ is finite; this shows that we have $\Delta(G/\Delta^+) = \Delta/\Delta^+$. Take some $x\in \Delta$, so that $x$ satisfies this condition: then, given arbitrary $g\in G$, there exists some $k$ such that $g^{p^k}\Delta^+ \in \mathbf{C}_{G/\Delta^+}(x\Delta^+)$, so that $(x^{-1}gx)^{p^k}\Delta^+ = g^{p^k}\Delta^+$, and it now follows from \cite[III, 2.1.4]{lazard} that $x^{-1}gx\Delta^+ = g\Delta^+$. This shows that $\Delta/\Delta^+ \leq Z(G/\Delta^+)$. Conversely, we must have $Z(G/\Delta^+) \leq \Delta(G/\Delta^+)$ by definition.
\end{proof}

We extend Theorem \ref{thm: control, nilpt p-valued} to:

\begin{propn}\label{propn: control, FNp}
Let $G$ be a finite-by-(nilpotent $p$-valuable) group and $k$ a finite field of characteristic $p$.

\begin{itemize}
\item[(i)] If $\mathfrak{p}$ is a $G$-prime ideal of $k\Delta$, then $\mathfrak{p}kG$ is a prime ideal of $kG$.
\item[(ii)] If $P$ is an almost faithful prime ideal of $kG$, then $P$ is controlled by $\Delta$.
\end{itemize}
\end{propn}

\begin{proof}
Adopt the notation of \cite[Lemma 1.1 and Notation 1.2]{woods-prime-quotients}. Let $e\in \cpi(\mathfrak{p})$, and write $f = \cpisum{e}{G}$. To prove (i), it suffices to prove that the ideal $f\cdot\overline{\mathfrak{p}kG}\lhd f\cdot\overline{kG}$ is prime. But, by the Matrix Units Lemma \cite[Lemma 6.1]{woods-prime-quotients}, we have an isomorphism
$$f\cdot\overline{kG} \cong M_s(e\cdot\overline{kG_1}),$$
where $G_1$ is the stabiliser in $G$ of $e$, and under which $f\cdot\overline{\mathfrak{p}kG} \mapsto M_s(e\cdot\overline{\mathfrak{p}_1 kG_1})$ for some $G_1$-prime ideal $\mathfrak{p}_1$ of $k[[\Delta\cap G_1]]$. So, by Morita equivalence, it will suffice to show that the ideal $e\cdot\overline{\mathfrak{p}_1 kG_1}\lhd e\cdot\overline{kG_1}$ is prime.

Now recall from \cite[Theorems A and C]{woods-prime-quotients} that we have an isomorphism
$$\psi: e\cdot\overline{kG_1} \cong M_t(k'[[G_1/\Delta^+]])$$
under which $e\cdot\overline{\mathfrak{p}_1 kG_1} \mapsto \mathfrak{q} k'[[G_1/\Delta^+]]$ for a $(G_1/\Delta^+)$-prime ideal $\mathfrak{q}$ of $k'[[\Delta\cap G_1/\Delta^+]]$. Hence we need now only show that $\mathfrak{q}k'N \lhd k'N$ is prime, where $N = G_1/\Delta^+$.

Note that, as $G_1$ is open in $G$, we have $\Delta(G_1) = \Delta\cap G_1$ \cite[Lemma 1.3(ii)]{woods-struct-of-G}; and from Lemma \ref{lem: centre = delta}, $\Delta(G_1)/\Delta^+ = Z(G_1/\Delta^+)$. Hence, still writing $N = G_1/\Delta^+$, we see that $\mathfrak{q}$ is an $N$-prime ideal of $k'[[Z(N)]]$, and hence a prime ideal. But now $\mathfrak{q}k'N$ is prime by Theorem \ref{thm: control, nilpt p-valued}(i). This establishes part (i) of the proposition.

To show part (ii), take an almost faithful prime ideal $P$ of $kG$. We would like to show that $P$ is a minimal prime ideal above $(P\cap k\Delta)kG$. But this is clearly true when $\Delta^+ = 1$ by Theorem \ref{thm: control, nilpt p-valued}; and in the general case, another application of the Matrix Units Lemma \cite[Lemma 6.1]{woods-prime-quotients} and \cite[Theorems A and C]{woods-prime-quotients}, as above, reduces to the case $\Delta^+ = 1$.

Hence, finally, we need only show that $(P\cap k\Delta)kG$ is prime; but $P\cap k\Delta$ is a $G$-prime ideal of $k\Delta$ (again by Lemma \ref{lem: G-prime stuff}(i)), so we are done by part (i) of the proposition.
\end{proof}

Until the end of this section, we will write $(-)^\circ$ to mean $\bigcap_{g\in G} (-)^g$.

\begin{cor}\label{cor: almost-FNp extends to FNp}
Let $G$ be a finite-by-(nilpotent $p$-valuable) group, and $H$ an open normal subgroup of $G$ containing $\Delta$. Let $k$ be a finite field of characteristic $p$. If $P$ is an almost faithful $G$-prime ideal of $kH$, then $PkG$ is a prime ideal of $kG$.
\end{cor}

\begin{proof}
Take a minimal prime $Q$ of $kH$ above $P$. Then we have $Q^\circ = P$, so $Q^\dagger$ is finite (as $G$ is orbitally sound \cite[Definition 1.4, Corollary 2.4]{woods-struct-of-G}). Hence $Q$ is controlled by $\Delta$, by Proposition \ref{propn: control, FNp}(ii), and by applying $(-)^\circ$ to both sides of the equality $Q = (Q\cap k\Delta)kH$, we see that $P$ is also: $P = (P\cap k\Delta)kH$. In particular $PkG = (P\cap k\Delta)kG$. But now Proposition \ref{propn: control, FNp}(i) shows that $(P\cap k\Delta)kG$ is prime.
\end{proof}

For the following results, we need to assume that $p > 2$ in order to be able to invoke \cite[Theorem A]{woods-extensions-of-primes}.

\begin{propn}\label{propn: af primes controlled by FNp}
Let $G$ be a nilpotent-by-finite, orbitally sound compact $p$-adic analytic group, and $k$ a finite field of characteristic $p > 2$. Let $H = \fn(G)$. If $P$ is an almost faithful prime ideal of $kG$, then $P$ is controlled by $H$.
\end{propn}

\begin{proof}
Let $Q$ be a minimal prime ideal of $kH$ above $P\cap kH$. Then $(Q^\dagger)^\circ = P^\dagger\cap H$ is finite, so, as $G$ is orbitally sound, $Q^\dagger$ is also finite. By \cite[Corollary 14.8]{passmanICP}, in order to prove that $(P\cap kH)kG$ is prime, it suffices to show that $QkS$ is prime, where $S$ is the stabiliser in $G$ of $Q$.

Let $T = \fn(S)$. As $H$ is a finite-by-(nilpotent $p$-valuable) open normal subgroup of $S$, we see that $H$ must be an open normal subgroup of $T$. It is also clear that $\Delta(H) = \Delta(T) = \Delta(S) = \Delta(G)$ \cite[Lemma 1.3(ii) and Theorem C]{woods-struct-of-G}. Now, by Corollary \ref{cor: almost-FNp extends to FNp}, $QkT$ must be prime; and we have that $(QkT)^\dagger$ is finite. Now, by the prime extension theorem \cite[Theorem A]{woods-extensions-of-primes}, $(QkT)kS = QkS$ is prime.
\end{proof}

\begin{lem}\label{lem: control, n-by-f, patching G - H - Delta}
Let $G$ be a nilpotent-by-finite compact $p$-adic analytic group, and let $H \geq K$ be any two closed normal subgroups of $G$. Take $P$ to be a prime ideal of $kG$. Let $Q$ be a minimal prime ideal of $kH$ above $P\cap kH$. If $P$ is controlled by $H$ and $Q$ is controlled by $K$, then $P$ is controlled by $K$.
\end{lem}

\begin{proof}
By Lemma \ref{lem: G-prime stuff}(ii), we have $Q^\circ = P\cap kH$, and so
\begin{align*}
(P\cap kK)kG &= ((P\cap kH)\cap kK)kG\\
&= (Q^\circ \cap kK)kG\\
&= (Q\cap kK)^\circ kG&\mbox{as } K \mbox{ is normal in } G\\
&= ((Q\cap kK)kH)^\circ kG&\mbox{as } H \mbox{ is normal in } G\\
&= Q^\circ kG &\mbox{as } Q \mbox{ is controlled by } K\\
&= (P\cap kH)kG = P &\mbox{as } P \mbox{ is controlled by } H.
\end{align*}
\end{proof}

Now back to:

\begin{thm}\label{thm: control, o.s. case}
Let $G$ be a nilpotent-by-finite, orbitally sound compact $p$-adic analytic group, $k$ a finite field of characteristic $p > 2$, and $P$ an almost faithful prime ideal of $kG$. Then $P$ is controlled by $\Delta$.
\end{thm}

\begin{proof}
Proposition \ref{propn: af primes controlled by FNp} shows that $P$ is controlled by $H$. Let $Q$ be a minimal prime of $kH$ above $P\cap kH$: then $Q^\circ = P\cap kH$ by Lemma \ref{lem: G-prime stuff}(ii), so we see that $(Q^\dagger)^\circ = P^\dagger\cap H$ is finite, so (as $G$ is orbitally sound) $Q^\dagger$ must also be finite. Hence, as $Q$ is almost faithful, Proposition \ref{propn: control, FNp}(ii) shows that it is controlled by $\Delta$. Now Lemma \ref{lem: control, n-by-f, patching G - H - Delta} applies.
\end{proof}

%\textit{Proof of Theorem \hyperref[thm: theorem K]{K}.} This follows from Theorem \ref{thm: control, o.s. case}.

%
%
%\subsection{a leftover lemma}
%
%\begin{lem}\label{lem: almost faithful primes 1}
%Let $G$ be a compact $p$-adic analytic group and $H$ a closed normal subgroup. Let $P$ be a prime ideal of $kG$, and denote reduction modulo $P^\dagger$ by $\overline{ (\cdot) }$. Suppose further that $P^\dagger \leq H$. Then $P$ is controlled by $H$ iff $\overline{P}$ is controlled by $\overline{H}$.
%\end{lem}
%
%\begin{proof}
%Extend the use of $\overline{(\cdot)}$ by writing it also for the ring homomorphism $kG \to kG/(P^\dagger - 1)kG$. As $(P^\dagger - 1)kG \subseteq P$, we have that
%$$(P\cap kH) + (P^\dagger - 1)kG = P \cap (kH + (P^\dagger - 1)kG)$$
%by the modular law, and so $\overline{P\cap kH} = \overline{P} \cap k\overline{H}$. Hence
%\begin{equation*}
%\begin{split}
%\overline{(P\cap kH)kG} &= (\overline{P\cap kH})\overline{kG} \\
%& = (\overline{P} \cap k\overline{H})k\overline{G}
%\end{split}
%\end{equation*}
%and so $(\overline{P} \cap k\overline{H})k\overline{G} = \overline{P}$ if and only if $(P\cap kH)kG = P$.
%\end{proof}
%

\subsection{Primes adjacent to faithful primes}

We begin with a property of the ``finite-by-(nilpotent $p$-valuable) radical" operator.

\begin{lem}\label{lem: quotients of FNp}
Let $G$ be a nilpotent-by-finite, orbitally sound compact $p$-adic analytic group, let $N$ be a normal subgroup of $G$ which is contained in $\Delta$, and let $F$ be a finite normal subgroup of $G$. Then the following three statements hold.
\begin{itemize}
\item[(i)] $\fn(G/F) = \fn(G)/F$.
\item[(ii)] Suppose that $\fn(G/\mathrm{i}_\Delta(N)) = \fn(G)/\mathrm{i}_\Delta(N)$. Then $\fn(G/N) = \fn(G)/N$.
\item[(iii)] Suppose $N$ is $\Delta$-isolated. Then we have either $\fn(G/N) = \fn(G)/N$ or $N = \Delta = \fn(G)$.
\end{itemize}
\end{lem}

\begin{proof}
$ $

\begin{itemize}
\item[(i)] This is clear from the construction of $\fn(G)$ (see \cite[Definition 5.3]{woods-struct-of-G}).
\item[(ii)] First, note that $\fn(G)/N$ is a quotient of a finite-by-(nilpotent $p$-valuable) normal subgroup of $G$, and hence is still a finite-by-(nilpotent $p$-valuable) normal subgroup of $G/N$, i.e. $\fn(G)/N \leq \fn(G/N)$. As both of these are of finite index in $G$, it will suffice to show that these indices are equal.

Consider the natural surjection $\alpha: G/N\to G/\mathrm{i}_\Delta(N)$. We can see that
$$\ker\alpha = \mathrm{i}_\Delta(N)/N = \Delta^+(\Delta/N) \leq \Delta^+(G/N)\leq \fn(G/N)$$
is a finite normal subgroup of $G/N$, and hence from (i) we see that
$$\frac{\fn(G/N)}{\mathrm{i}_\Delta(N)/N} \cong \fn(G/\mathrm{i}_\Delta(N)).$$
That is, the restricted map $$\alpha|_{\fn(G/N)} : \fn(G/N) \to \fn(G/\mathrm{i}_\Delta(N))$$ is also surjective with kernel $\mathrm{i}_\Delta(N)/N$. Hence we have the following commutative diagram, in which the first two rows are exact, all three columns are exact, and $C_1$ and $C_2$ are the cokernels of the vertical maps.

\centerline{
\xymatrix{
& 1 \ar[d] & 1 \ar[d] & 1 \ar[d] &\\
1 \ar[r] & \mathrm{i}_\Delta(N)/N \ar[r]\ar@{=}[d] & \fn(G/N) \ar[r]\ar[d] & \fn(G/\mathrm{i}_\Delta(N)) \ar[r]\ar[d] & 1\\
1 \ar[r] & \mathrm{i}_\Delta(N)/N \ar[r]\ar[d] & G/N \ar[r]\ar[d] & G/\mathrm{i}_\Delta(N) \ar[r]\ar[d]& 1\\
1 \ar[r] & 1 \ar[r]\ar[d] & C_1 \ar[r]\ar[d] & C_2 \ar[r]\ar[d] & 1\\
& 1 & 1 & 1 &
}
}

By the Nine Lemma \cite[Chapter XII, Lemma 3.4]{maclane}, the third row is now also exact, so that $C_1 \cong C_2$. But by assumption, $C_2 \cong G/\fn(G)$, and hence
$$[G/N: \fn(G/N)] = |C_1| = [G: \fn(G)] = [G/N, \fn(G)/N],$$
as required.
\item[(iii)]
\textbf{Case 1.} First, assume that $\Delta^+ = 1$.

Write $H = \fn(G)$, and $\widehat{H}$ for the preimage of $\widehat{H}/N = \fn(G/N)$.

If $G = \fn(G)$, then we clearly have $\fn(G/N) = \fn(G)/N$ for any closed normal subgroup $N$. So suppose that $H\lneq \widehat{H} \leq G$, and take some $z\in \widehat{H}\setminus H$. Now conjugation by $z$ induces the automorphism $x\mapsto x^\zeta$ on $H/H'$ (where $H'$ denotes the isolated derived subgroup), and hence also on $H/H'N$, for some $\zeta\in t(\mathbb{Z}_p^\times)$ \cite[Lemma 4.2]{woods-struct-of-G} satisfying $\zeta \neq 1$ \cite[Lemma 3.3]{woods-extensions-of-primes}.

If $H/H'N$ has nonzero rank, we may take an element $x\in H$ whose image in $H/H'N$ has infinite order; and now the image in $\widehat{H}/H'N$ of $\overline{\langle x,z \rangle}$ is not finite-by-nilpotent, contradicting the definition of $\widehat{H}$. So we must have $H = \mathrm{i}_H(H'N)$.

In particular, this implies that $H = \mathrm{i}_H(H'Z)$, where $Z = Z(H) = \Delta(G)$, and so, by \cite[Lemma 3.5]{woods-extensions-of-primes}, we see that $H$ is abelian, i.e. $H = \Delta$. Furthermore, this implies that $H' = 1$, and as $N$ is already $H$-isolated (because $\Delta$ is $H$-isolated), the equality $H = \mathrm{i}_H(H'N)$ simplifies to give $H = N$. This is what we wanted to prove.

\textbf{Case 2.} Now suppose instead that $\Delta^+ \neq 1$. As $N$ is isolated in $G$, we see that
\begin{itemize}
\item $\Delta^+ \leq N$, and $N/\Delta^+$ is isolated normal inside $G/\Delta^+$, contained in $\Delta/\Delta^+$;
\item $\Delta^+(G/\Delta^+) = 1$;
\item $\Delta(G/\Delta^+) = \Delta/\Delta^+ = Z(\fn(G)/\Delta^+)$;
\item $\fn(G/\Delta^+) = \fn(G)/\Delta^+$;
\end{itemize}
and so the result follows by applying Case 1 to $G/\Delta^+$.\qedhere
\end{itemize}
\end{proof}

\begin{rk}
If $G$ is a compact $p$-adic analytic group, $H$ is a closed normal subgroup, and $Q$ is a $G$-stable ideal of $kH$, then $Q^\dagger = (Q+1)\cap H$ is normal in $G$.
\end{rk}

\begin{lem}\label{lem: extension of G-prime from kDelta to kG is prime}
Let $G$ be a nilpotent-by-finite, orbitally sound compact $p$-adic analytic group, and let $k$ be a finite field of characteristic $p>2$. If $Q$ is a $G$-prime ideal of $k\Delta$, and $\fn(G/Q^\dagger) = \fn(G)/Q^\dagger$, then $QkG$ is a prime ideal of $kG$.
\end{lem}

\begin{rk}
The hypothesis
\begin{align}\label{eqn: hypothesis}
\fn(G/Q^\dagger) = \fn(G)/Q^\dagger\tag{$\ddagger$}
\end{align}
has the following consequence. Let $G$ be a nilpotent-by-finite, orbitally sound compact $p$-adic analytic group, $k$ a finite field of characteristic $p>2$, and let $P\lneq P'$ be adjacent prime ideals of $kG$, with $P$ almost faithful. Then $P$ is controlled by $\Delta$, by Theorem \ref{thm: control, o.s. case}. Set $Q := P'\cap k\Delta$. Consider $\mathrm{i}_\Delta(Q^\dagger)$: if this is not equal to $\Delta$, then by Lemma \ref{lem: quotients of FNp}(ii), (iii), the hypothesis (\ref{eqn: hypothesis}) is satisfied. So suppose it is equal to $\Delta$. Now, as $Q$ contains the ideal $\ker(kG\to k[[G/Q^\dagger]])$ (the augmentation ideal of $Q^\dagger$), if we further have that $\fn(G) = \Delta$, then $kG/Q$ is a \emph{finite} prime ring, which is therefore simple, and so $Q$ must be a \emph{maximal} ideal of $kG$ of $\mathrm{i}_G(\Delta) = G$; otherwise, we again have (\ref{eqn: hypothesis}) by Lemma \ref{lem: quotients of FNp}(ii), (iii).

That is, under these conditions, we always have (\ref{eqn: hypothesis}) unless $Q$ is a maximal ideal of $kG$ and $G$ is virtually abelian, in which case $Q^\dagger$ is open in $G$.
\end{rk}

\begin{proof}
Write $H = \fn(G)$.

As $Q$ is a $G$-prime, we may write it as $\bigcap_{g\in G} I^g$ for some minimal prime ideal $I$ above $Q$. Suppose the $G$-orbit of $I$ splits into distinct $H$-orbits $\mathcal{O}_1, \dots, \mathcal{O}_r$, and write $P_i := \bigcap_{A\in \mathcal{O}_i} A$. Then $P_i$ is an $H$-prime of $k\Delta$, and $\bigcap_{i=1}^r P_i = Q$. In particular, since $P_i$ is an $H$-prime of $k\Delta$, we have that $P_i kH$ is prime by Proposition \ref{propn: control, FNp}(i).

It remains to show that $\left(\bigcap_{g\in G} (P_i kH)^g\right) kG$ is prime. By \cite[Corollary 14.8]{passmanICP}, it suffices to show that $P_ikS$ is prime, where $S = \mathrm{Stab}_G(P_i)$.

Write $\mathfrak{p} = P_ikH$, and note that $\mathfrak{p}^\dagger = P_i^\dagger \leq \Delta$. Now, if $\fn(G)/\Delta^+$ is non-abelian, we have $\fn(S/\mathfrak{p}^\dagger) = \fn(S)/\mathfrak{p}^\dagger$. If, on the other hand, $\fn(G)/\Delta^+$ is abelian, then we must have $Q^\dagger \lneq \Delta$, and as $Q^\dagger$ is $H$-isolated orbital, we have $[\Delta: Q^\dagger] = \infty$. But as $G$ is orbitally sound, and $Q^\dagger = \bigcap_{g\in G} (P_i^\dagger)^g$, we must have that $Q^\dagger$ is open in $P_i^\dagger$, so that in particular $[\Delta: \mathfrak{p}^\dagger] = \infty$. Hence again we have $\fn(S/\mathfrak{p}^\dagger) = \fn(S)/\mathfrak{p}^\dagger$.

Write $\overline{(\cdot)}$ for the quotient map $S\to S/\mathfrak{p}^\dagger$. Now, to show that $P_ikS = \mathfrak{p}kS$ is prime, we need only show that $\overline{\mathfrak{p}}k\overline{S}$ is prime. But $\overline{\mathfrak{p}}$ is a faithful prime ideal of $k\overline{H}$, and $\overline{H} = \fn(\overline{S})$, so by \cite[Theorem A]{woods-extensions-of-primes}, we are done.
\end{proof}

\begin{lem}\label{lem: Q also controlled by Delta}
Let $k$ be a finite field of characteristic $p>2$. Let $G$ be a nilpotent-by-finite, orbitally sound compact $p$-adic analytic group, and let $P \lneq Q$ be adjacent prime ideals of $kG$, with $P$ almost faithful. Suppose that $Q$ is not a maximal ideal of $kG$. Then $Q$ is controlled by $\Delta$.
\end{lem}

\begin{proof}
$Q\cap k\Delta$ is a $G$-prime of $k\Delta$, and so $(Q\cap k\Delta)kG$ is prime by Lemma \ref{lem: extension of G-prime from kDelta to kG is prime} and the accompanying remark. But
$$P = (P\cap k\Delta)kG \leq (Q\cap k\Delta)kG \leq Q,$$
(with the equality as a result of Theorem \ref{thm: control, o.s. case}), and $P$ and $Q$ are adjacent, so $(Q\cap k\Delta)kG$ must equal either $P$ or $Q$.

Let us assume for contradiction that $(Q\cap k\Delta)kG = P$. Then we must have
$$P\cap k\Delta \leq Q\cap k\Delta \leq (Q\cap k\Delta)kG = P,$$
and by intersecting each of these with $k\Delta$, we see that $P\cap k\Delta = Q\cap k\Delta$. In particular, by taking $(\cdot)^\dagger$ of both sides of this equality, we see that $Q^\dagger\cap \Delta$ is finite (as $P$ is almost faithful).

Let $N$ be an open normal nilpotent $p$-valued subgroup of $G$, and let $Z = Z(N)$. By \cite[Lemma 1.3(ii)]{woods-struct-of-G}, $Z = \Delta(N)$ is a finite-index torsion-free subgroup of $\Delta$, and so $Q^\dagger \cap Z = 1$. Now, as $N$ is nilpotent and the normal subgroup $Q^\dagger \cap N$ has trivial intersection with its centre, \cite[5.2.1]{rob} implies that $Q^\dagger \cap N = 1$, and hence $Q^\dagger$ must be a finite normal subgroup of $G$. So $Q^\dagger \leq \Delta^+$, and in particular $Q^\dagger = Q^\dagger\cap\Delta$, which we earlier determined is finite. Hence $Q$ is almost faithful, and must be controlled by $\Delta$ by Theorem \ref{thm: control, o.s. case}. In particular, we must have $P\cap k\Delta \neq Q\cap k\Delta$. But this contradicts our assumption.
\end{proof}

\section{Catenarity}\label{section: catenarity}

\subsection{The orbitally sound case: plinths and a height function}

Much of the material in this subsection is adapted from \cite{roseblade}.

Unless stated otherwise, throughout this section, $G$ is an arbitrary compact $p$-adic analytic group, and $k$ is a finite field of characteristic $p$. We start by outlining our plan of attack:

\begin{lem}\label{lem: nice height function implies catenary}
Let $R$ be a ring in which every prime ideal has finite height. Suppose we are given a function $h: \mathrm{Spec}(R)\to \mathbb{N}$ satisfying
\begin{itemize}
\item $h(P) = 0$ whenever $P$ is a minimal prime of $R$,
\item $h(P') = h(P) + 1$ for each pair of adjacent primes $P\lneq P'$ of $R$.
\end{itemize}
Then $R$ is a catenary ring.
\end{lem}

\begin{proof}
Obvious.
\end{proof}

\begin{lem}
$kG$ has finite classical Krull dimension, i.e. the maximal length of any chain of prime ideals is bounded.
\end{lem}

\begin{proof}
The classical Krull dimension of $kG$ is bounded above by $\Kdim(kG)$ by \cite[Lemma 6.4.5]{MR}, which is equal to $\Kdim(\mathbb{F}_p G)$ by \cite[Proposition 6.6.16(ii)]{MR}, and this is bounded above by the \emph{dimension} (in the sense of \cite[Theorem 8.36]{DDMS}) of $G$, which is finite by definition (see the remarks after \cite[Definition 3.12]{DDMS}).
\end{proof}

\begin{defn}\label{defn: plinths}
Let $V$ be a $\mathbb{Q}_p G$-module, and suppose it has finite dimension as a vector space over $\mathbb{Q}_p$. Take a chain $$0 = V_0 \lneq V_1 \lneq \dots \lneq V_r = V$$ of $G$--orbital subspaces -- that is, $\mathbb{Q}_p$-vector subspaces of $V$ with finitely many $G$-conjugates, or equivalently $\mathbb{Q}_p$-vector subspaces that are $\mathbb{Q}_p N$-submodules for some open subgroup $N$ of $G$. Assume further that this chain is \emph{saturated}, in the sense that it cannot be made longer by the addition of some $G$--orbital subspace $V_i \lneq V' \lneq V_{i+1}$. Such a chain is necessarily finite, as it is bounded above in length by $\dim_{\mathbb{Q}_p}(V) + 1$. We call the number $r$ the \emph{$G$--plinth length} of $V$, written $p_G(V)$. If $p_G(V) = 1$, we say that $V$ is a \emph{plinth} for $G$.
\end{defn}

\begin{rk}
The number $r$ is independent of the $V_i$ chosen. Indeed, fix a \emph{longest possible} chain $$0 = V_0 \lneq V_1 \lneq \dots \lneq V_r = V$$ of $G$-orbital subspaces, and let $G_0$ be the intersection of the normalisers $\mathbf{N}_G(V_i)$, i.e. the largest subgroup of $G$ such that each $V_i$ is a $\mathbb{Q}_p G_0$-module. $G_0$ is open in $G$. Now, given any chain $$0 = W_0 \lneq W_1 \lneq \dots \lneq W_s = V$$ of $G$-orbital subspaces, take $H_0 = \bigcap_{j=1}^s \mathbf{N}_G(W_j)$, and note that $G_0\cap H_0$ is a finite-index open subgroup of $G$ that normalises each $V_i$ and $W_j$. Hence, by the Jordan-H\"older theorem \cite[Theorem 4.11]{GW}, the chain $W_j$ may be refined to a chain of length $r$; so if the chain $W_j$ is saturated, then $s = r$.
\end{rk}

\begin{defn}\label{defn: G-groups}
A \emph{$G$-group} is a topological group $H$ endowed with a continuous action of $G$. For example, closed subgroups of $G$, and quotients of $G$ by closed normal subgroups of $G$, are $G$-groups under the action of conjugation.

Let $H$ be a nilpotent-by-finite compact $p$-adic analytic group with a continuous action of $G$. We aim to define $p_G(H)$. In fact, as plinths are insensitive to finite factors, we may immediately replace $H$ by the open subgroup formed by the intersection of the (finitely many) $G$-conjugates of any given open normal nilpotent uniform subgroup of $H$. Then there is a series
\begin{align}\label{eqn: H series thing}
1 = H_0 \lhd H_1 \lhd \dots \lhd H_n = H
\end{align}
of $G$-subgroups such that $A_i = H_i/H_{i-1}$ is abelian for each $i = 1, \dots, n$. Let $V_i = A_i \tensor{\mathbb{Z}_p} \mathbb{Q}_p$ for each $i = 1, \dots, n$, with $G$--action given by conjugation. In this case, we define $$p_G(H) = \sum_{i=1}^n p_G(V_i).$$
\end{defn}

\begin{lem}
$p_G(H)$ is well-defined, and does not depend on the series (\ref{eqn: H series thing}).
\end{lem}

\begin{proof}
Apply the Jordan-H\"older theorem, as in the remark above.
\end{proof}

For our purposes, the most important property of $p_G$ is that it is additive on short exact sequences of $G$-groups, which also follows from the Jordan-H\"older theorem. We record this as:

\begin{lem}\label{lem: plinth length is additive on ses}
Suppose that $\xymatrix{1\ar[r]& A\ar[r]& B\ar[r] & C\ar[r]& 1}$ is a short exact sequence of $G$-groups. Then $p_G(A) + p_G(C) = p_G(B)$.\qed
\end{lem}

We now define Roseblade's function $\lambda$. (Later, we will show that, in the case when $G$ is nilpotent-by-finite and orbitally sound, $\lambda$ is actually equal to the height function on $\mathrm{Spec} (kG)$.)

\begin{defn}\label{defn: lambda}
$$
\lambda(P) = \left\{
\begin{array}{ll}
p_G(P^\dagger) + \lambda(P^\pi) & P^\dagger \neq 1 \\
h_G(P\cap k\Delta)& P^\dagger = 1,
\end{array} \right.
$$
where $P^\pi$ is the image of $P$ under the map $$\pi: kG \to kG/(P^\dagger - 1)kG \cong k[[G/P^\dagger]].$$
\end{defn}

This definition is recursive, in that if $P$ is an unfaithful prime ideal, then $\lambda(P)$ is defined with reference to $\lambda(P^\pi)$; but $P^\pi$ is then a faithful prime ideal of $kG^\pi$, so this process terminates after at most two steps.

We make the following remark on this definition immediately:

\begin{lem}\label{lem: if P faithful, then lambda = h}
Let $G$ be a nilpotent-by-finite, orbitally sound compact $p$-adic analytic group, $k$ a finite field of characteristic $p>2$, and $P$ a faithful prime of $k\Delta$. Then $\lambda(P) = h(P)$.
\end{lem}

\begin{proof}
$\lambda(P)$ is defined to be $h_G(P\cap k\Delta)$. But, by Theorem \ref{thm: control, o.s. case} and Lemma \ref{lem: extension of G-prime from kDelta to kG is prime}, we see that there is a one-to-one, inclusion-preserving correspondence between faithful prime ideals of $kG$ and faithful $G$-prime ideals of $k\Delta$, so that $h_G(P\cap k\Delta) = h(P)$.
\end{proof}

We return to the general case of $G$ an arbitrary compact $p$-adic analytic group.

\begin{lem}\label{lem: assume wlog P is faithful for lambda}
Let $P \lneq Q$ be neighbouring prime ideals of $kG$, and write $$\pi: kG \to kG/(P^\dagger - 1)kG \cong k[[G/P^\dagger]].$$ Then $$\lambda(Q) - \lambda(P) = \lambda(Q^\pi) - \lambda(P^\pi).$$
\end{lem}

\begin{proof}
Firstly, as $P \leq Q$, we have $P^\dagger \leq Q^\dagger$, so the map $$\rho: kG \to k[[G/Q^\dagger]]$$ factors as

\centerline{
\xymatrix{
kG\ar@{->}[r]^-{\pi}\ar@{->}@/_1pc/[rr]_{\rho}& k[[G/P^\dagger]]\ar@{->}[r]^{\sigma}& k[[G/Q^\dagger]].
}
}

We now compute $\lambda(Q) - \lambda(P)$ using Definition \ref{defn: lambda}:
\begin{align*}
\lambda(Q) - \lambda(P) &= p_G(Q^\dagger) - p_G(P^\dagger) + \lambda(Q^\rho) - \lambda(P^\pi) & \\
&= p_G((Q^\dagger)^\pi) + \lambda(Q^\rho) - \lambda(P^\pi) & \mbox{by Lemma \ref{lem: plinth length is additive on ses}} \\
&= p_G((Q^\dagger)^\pi) + \lambda(Q^{\pi\sigma}) - \lambda(P^\pi) & \mbox{by definition of } \rho \\
&= p_G((Q^\pi)^\dagger) + \lambda((Q^\pi)^\sigma) - \lambda(P^\pi) & \mbox{as } (Q^\dagger)^\pi = (Q^\pi)^\dagger \\
&= \lambda(Q^\pi) - \lambda(P^\pi) & \mbox{by Definition \ref{defn: lambda}.\!\!}&\qedhere
\end{align*}
\end{proof}

\begin{rk}
Suppose $G$ is a nilpotent-by-finite compact $p$-adic analytic group, and suppose we are given a subquotient $A$ of $G$ which is a plinth, with $G$-action induced from the conjugation action of $G$ on itself. Then it is easy to see that $\dim_{\mathbb{Q}_p}(A\tensor{\mathbb{Z}_p} \mathbb{Q}_p) = 1$. (Roseblade calls such plinths \emph{centric}.) Indeed, suppose $A = H/K$, where $H$ and $K$ are closed normal subgroups of $G$ with $K$ contained in $H$. Then we may replace $G$ by an open normal nilpotent uniform subgroup $G'$, and $A$ by $A' = H'/K'$, where $H' = H\cap G'$ and $K' =\mathrm{i}_{H'} (K\cap G')$; after doing this, we still have that $A'$ is a plinth for $G'$, and that $\dim_{\mathbb{Q}_p}(A\tensor{\mathbb{Z}_p} \mathbb{Q}_p) = \dim_{\mathbb{Q}_p}(A'\tensor{\mathbb{Z}_p} \mathbb{Q}_p)$. But, as $G'/K'$ is nilpotent, and $A'$ is a non-trivial normal subgroup, $A'$ must meet the centre $Z(G'/K')$ non-trivially; and as $A'$ is torsion-free, we must have that $A'\cap Z(G'/K')$ is a plinth for $G'$, and so must be equal to $A'$. Hence $G'$ centralises $A'$, and its plinth length is simply equal to its rank.
\end{rk}

Again, we will write $(-)^\circ$ to mean $\bigcap_{g\in G} (-)^g$.

\begin{lem}\label{lem: U is a G-prime}
Let $G$ be a nilpotent-by-finite compact $p$-adic analytic group. Let $U$ be a $G$-prime of $k\Delta$, and write $\rho: k\Delta \to k[[\Delta/U^\dagger]]$. Then $h(U) = h_G(U^\rho) + p_G(U^\dagger)$.
\end{lem}

\begin{proof}
Let $A = Z(\Delta)$, and let $U_1$ be a minimal prime of $kA$ above $U\cap kA$, so that $U\cap kA = U_1^\circ$. Then $h_G(U) = h(U_1)$ by Corollary \ref{cor: height and G-height of nearby things}, and so $h_G(U^\rho) = h(U_1^\rho)$. Now, from Lemma \ref{lem: commutative prime facts}(i), we have $h(U_1) + \dim(kA/U_1) = r(A)$ and $h(U_1^\rho) + \dim(kA/U_1) = r(A^\rho)$, from which we may deduce that $$h(U_1) = h(U_1^\rho) + r(A) - r(A^\rho).$$ But $r(A) - r(A^\rho) = r(A\cap \ker \rho) = p_G(U^\dagger \cap A)$ by the above remark. Now this is just $p_G(U^\dagger)$, as $A$ is open in $\Delta$.
\end{proof}

\begin{lem}\label{lem: lemma 29}
Let $G$ be arbitrary compact $p$-adic analytic. Let $H$ be a closed normal subgroup of $G$, and let $K$ be an open subgroup of $H$ which is normal in $G$. If $P$ is a $G$-prime ideal of $kH$, then $h_G(P) = h_G(P\cap kK)$.
\end{lem}

\begin{proof}
{}[Adapted from \cite[Lemma 29]{roseblade}.]
We know that $P = Q^\circ$ for some prime $Q$ of $kH$, and $Q\cap kK = \bigcap_{h\in H} V^h$ for some prime $V$ of $kK$. Hence $P\cap kK = V^\circ$. Then, writing $h_G^{\mathrm{orb}}$ for the height function on \emph{$G$-orbital primes},
\begin{align*}
h_G(P) &= h_G^{\mathrm{orb}}(Q) & \text{by Lemma \ref{lem: G-prime stuff}(ii)}\\
&= h_G^{\mathrm{orb}}(V) & \text{by Lemma \ref{lem: G-prime stuff}(i)}\\
&= h_G(P\cap kK) & \text{by Lemma \ref{lem: G-prime stuff}(ii)}.\!\!&\qedhere
\end{align*}
\end{proof}

Here, we deduce from Theorem \ref{thm: control, o.s. case} and \cite[proof of theorem H2]{roseblade} the following corollary:

\begin{thm}\label{thm: catenarity in o.s. case}
Let $G$ be a nilpotent-by-finite, orbitally sound compact $p$-adic analytic group, and $k$ a finite field of characteristic $p > 2$. Then $kG$ is a catenary ring.
\end{thm}

\begin{proof}
Let $P \lneq Q$ be neighbouring prime ideals of $kG$. We will first show that $\lambda(Q) = \lambda(P) + 1$.

By passing to $k[[G/P^\dagger]]$, we may assume that $P$ is a faithful prime ideal, by Lemma \ref{lem: assume wlog P is faithful for lambda}. Hence, by Theorem \ref{thm: control, o.s. case} (and as $p > 2$), we have that $(P\cap k\Delta)kG = P$. We also have either that $(Q\cap k\Delta)kG = Q$, by Lemma \ref{lem: Q also controlled by Delta}, or $Q^\dagger \geq \Delta$ by the remark of Lemma \ref{lem: extension of G-prime from kDelta to kG is prime}; and so, in either case, we have $P\cap k\Delta \lneq Q\cap k\Delta$.

We will now show that $P\cap k\Delta$ and $Q\cap k\Delta$ are \emph{neighbouring} $G$-primes of $k\Delta$. Suppose that they are not: then there must be a $G$-prime $J$ strictly between them, i.e. $P\cap k\Delta\lneq J\lneq Q\cap k\Delta$. Then, again by Lemma \ref{lem: extension of G-prime from kDelta to kG is prime} and the remark made there, we see that $JkG$ is a prime ideal of $kG$. Now it is clear that $P\leq JkG\leq Q$ by the previous paragraph; and if $JkG = Q$, then intersecting both sides with $k\Delta$ shows that $J = Q\cap k\Delta$, and likewise if $JkG = P$. Hence we must have $P\lneq JkG \lneq Q$, so that $P$ and $Q$ are \emph{not} neighbouring primes. But this contradicts our initial assumptions.

So we conclude that $h_G(Q\cap k\Delta) = h_G(P\cap k\Delta) + 1$. The right hand side is, by definition, just equal to $\lambda(P) + 1$; and we have $\lambda(Q) = \lambda(Q^\rho) + p_G(Q^\dagger)$, where $\rho: G\to G/Q^\dagger$. It remains to show that this is equal to $h_G(Q\cap k\Delta)$.

\textbf{Case 1.} $Q^\dagger$ is not open in $G$. Then $Q$ is controlled by $\Delta$ by Lemma \ref{lem: Q also controlled by Delta} and the remark of Lemma \ref{lem: extension of G-prime from kDelta to kG is prime}, and so $Q^\rho$ is controlled by $\Delta^\rho$, and in particular by $\mathrm{i}_{G^\rho}(\Delta^\rho)\leq \mathrm{i}_{G^\rho}(\Delta(G^\rho))$. Write $A = Z(\Delta(G^\rho))$ and $B = A\cap \mathrm{i}_{G^\rho}(\Delta^\rho)$: as $Q^\rho$ is controlled by $\mathrm{i}_{G^\rho}(\Delta^\rho)$, we have that $Q^\rho\cap kA$ is controlled by $B$. Furthermore, we can write $Q^\rho\cap kA = \mathfrak{q}^\circ$ for some prime $\mathfrak{q}$ of $kA$, so that $\mathfrak{q}$ is also controlled by $B$, and hence
\begin{align*}
\lambda(Q^\rho) &= h_G(Q^\rho \cap k[[\Delta(G^\rho)]]) & \text{by definition}\\
&= h_G(Q^\rho \cap kA) & \text{by Lemma \ref{lem: lemma 29}}\\
&= h(\mathfrak{q}) & \text{by Corollary \ref{cor: height and G-height of nearby things}}\\
&= h(\mathfrak{q}\cap kB) & \text{by Lemma \ref{lem: commutative prime facts}(iii)}\\
&= h_G(Q^\rho \cap kB) & \text{by Corollary \ref{cor: height and G-height of nearby things}}\\
&= h_G(Q^\rho \cap k[[\mathrm{i}_{G^\rho}(\Delta^\rho)]]) & \text{by Lemma \ref{lem: lemma 29}}\\
&= h_G(Q^\rho \cap k\Delta^\rho) & \text{by Lemma \ref{lem: lemma 29}}.
\end{align*}

We also have $p_G(Q^\dagger) = p_G((Q\cap k\Delta)^\dagger)$. Hence
$$\lambda(Q) = h_G((Q\cap k\Delta)^\rho) + p_G((Q\cap k\Delta)^\dagger).$$
Now we are done by Lemma \ref{lem: U is a G-prime}.

\textbf{Case 2.} $Q^\dagger$ is open in $G$. We have already seen that this case only occurs when $G = \mathrm{i}_G(\Delta)$, and so $\lambda(Q^\rho) = \lambda({0}) = 0$, and $p_G(Q^\dagger) = p_G(G)$, and $h_G(Q\cap k\Delta) = h_G(Q\cap kA)$ $= r(A)$. These are clearly equal, as $A$ is open in $G$.

In order to invoke Lemma \ref{lem: nice height function implies catenary}, it remains only to show that $\lambda(P) = 0$ when $P$ is a minimal prime. But as all minimal primes are induced from $\Delta^+$, this follows immediately from the definition of $\lambda$: we will have $P^\dagger\leq \Delta^+$ (and hence $p_G(P^\dagger) = 0$), and $P^\pi\cap k\Delta^\pi$ will be a minimal $G$-prime of $k\Delta^\pi$ (and hence $h_G(P^\pi\cap k\Delta^\pi) = 0$).
\end{proof}

\subsection{Vertices and sources}

We now study a more general setting. Let $G$ be an arbitrary compact $p$-adic analytic group, and $P$ an arbitrary prime ideal of $kG$.

\begin{rk}
Suppose $G$ is orbitally sound and nilpotent-by-finite, $N$ is a closed normal subgroup of $G$, and $I$ is a prime ideal of $kG$ with $N \leq I^\dagger$ and $[I^\dagger : N]<\infty$. Writing $\overline{(\cdot)}$ for the natural map $kG\to k[[G/N]]$, it is clear that the prime ideal $\overline{I}\lhd k[[G/N]]$ is almost faithful, and so, by Theorem \ref{thm: control, o.s. case}, is controlled by $\Delta(G/N)$, and that $I$ is the complete preimage in $kG$ of $\overline{I}$, and is therefore controlled by the preimage in $G$ of $\Delta(G/N)$.
\end{rk}

This motivates the following definition:

\begin{defn}
Let $I$ be an ideal of $kG$, and $N$ a closed subgroup of $G$. We say that $I$ is \emph{almost faithful mod $N$} if $I^\dagger$ contains $N$ as a subgroup of finite index. We also write $\nabla_G(N)$ for the subgroup of $\mathbf{N}_G(N)$ defined by $$\nabla_G(N) / N = \Delta(\mathbf{N}_G(N)/N).$$ Diagrammatically:

\centerline{
\xymatrix{
G\ar@{-}[d]\\
\mathbf{N}_G(N)\ar@{-}[dd]\ar@{..}[r]& \mathbf{N}_G(N)/N\ar@{-}[dd]\\
\\
\nabla_G(N)\ar@{-}[d]\ar@{..}[r]& \Delta(\mathbf{N}_G(N)/N)\ar@{-}[d]\\
N\ar@{-}[d]\ar@{..}[r]& N/N\\
1
}
}
\end{defn}

We will extend this notion to ideals $I$ with $I^\dagger$ \emph{contained in} $N$ as a subgroup of finite index.

\begin{lem}\label{lem: open char subgroup}
Let $H$ be an open subgroup of $N$. Then there exists an open characteristic subgroup $M$ of $N$ contained in $H$.
\end{lem}

\begin{proof}
(Adapted from \cite[19.2]{passmanICP}.) Let $[N:H] = n < \infty$. Now, as $N$ is topologically finitely generated, there are only finitely many continuous homomorphisms $N\to S_n$. where $S_n$ is the symmetric group. Take $M$ to be the intersection of the kernels of these homomorphisms.
\end{proof}

\begin{lem}\label{lem: sub N = mod M}
Let $N$ be a closed subgroup of $G$, and $A = \mathbf{N}_G(N)$. Suppose $I$ is an ideal of $kA$, and $I^\dagger \leq N$ with $[N:I^\dagger]<\infty$. Then there is a closed normal subgroup $M$ of $A$ such that $I$ is almost faithful mod $M$. Furthermore, this $M$ can be chosen so that $\nabla_G(N) = \nabla_A(M)$.
\end{lem}

\begin{proof}
Set $H = I^\dagger$ in Lemma \ref{lem: open char subgroup}: then the subgroup $M$ is characteristic in $N$, hence normal in $A$; $M$ contains $I^\dagger$; and $M$ is open in $N$, so we must have $[I^\dagger : M] < \infty$.

By definition, we have $\nabla_G(N) = \nabla_A(N)$. Now, $N/M$ is a finite normal subgroup of $A/M$, so is contained in $\Delta^+(A/M)$. Hence the preimage under the natural quotient map $A/M\to A/N$ of $\Delta(A/N)$ is $\Delta(A/M)$. But this is the same as saying that $\nabla_A(N) = \nabla_A(M)$.
\end{proof}

%%%%%
\iffalse
\begin{lem}
Let $G$ be a nilpotent-by-finite compact $p$-adic analytic group, and $N$ an orbitally sound closed subgroup of $G$. Let $A = \mathbf{N}_G(N)$, and let $I$ be a prime ideal of $kA$ with $I^\dagger \leq N$, $[N:I^\dagger]<\infty$. Then $I$ is controlled by $\nabla_G(N)$.
\end{lem}

\begin{proof}
Taking $M$ as in Lemma \ref{lem: sub N = mod M}, we see that $I$ is almost faithful mod $M$, and hence is controlled by $\nabla_A(M) = \nabla_G(N)$, by the remark at the beginning of this subsection.
\end{proof}
\fi
%%%%%

When $G$ is a general compact $p$-adic analytic group, we will use the following lemma to translate between prime ideals of $kG$ and prime ideals of $kA$ for certain open subgroups $A$ of $G$.

\begin{lem}\label{lem: induced ideals from kA}
Let $H$ be an open normal subgroup of $G$. Suppose $P$ is a prime of $kG$, and write $Q$ for a minimal prime of $kH$ above $P\cap kH$. Let $B$ be the stabiliser in $G$ of $Q$, and let $A$ be any open subgroup of $G$ containing $B$, so that
$$H\leq B\leq A\leq G.$$
Then there is a prime ideal $T$ of $kA$ with $P = T^G$, and furthermore this $T$ satisfies $T\cap kH = \bigcap_{a\in A} Q^a$.
\end{lem}

\begin{proof}
This follows from \cite[14.10(i)]{passmanICP}.
\end{proof}

\begin{defn}
A prime $P\lhd kG$ is \emph{standard} if it is controlled by $\Delta$ and we have $P\cap k\Delta = \bigcap_{x\in G} L^x$ for some almost faithful prime $L\lhd k\Delta$.
\end{defn}

\begin{lem}\label{lem: T standard iff Q standard}
Let $G$ be a nilpotent-by-finite compact $p$-adic analytic group and $H$ an open normal subgroup. Let $P$ be a prime ideal of $kG$, and $Q$ a minimal prime of $kH$ above $P$, so that $P\cap kH = \bigcap_{x\in G} Q^x$. If $Q$ is a standard prime, then $P$ is a standard prime.
\end{lem}

\begin{proof}
(Adapted from \cite[20.4(i)]{passmanICP}.)
Write $\Delta = \Delta(G), \Delta_H = \Delta(H)$, and
$$P\cap k\Delta = \bigcap_{x\in G} S^x\qquad \mbox{and} \qquad Q\cap k\Delta_H = \bigcap_{y\in H} T^y,$$
for prime ideals $S\lhd k\Delta$ and $T\lhd k\Delta_H$. On the one hand,
\begin{align*}
P\cap k\Delta_H &=(P\cap k\Delta) \cap k\Delta_H\\
&= \left(\bigcap_{x\in G} S^x\right)\cap k\Delta_H\\
&=\bigcap_{x\in G}\left( S\cap k\Delta_H\right)^x,
\end{align*}
but on the other hand,
\begin{align*}
P\cap k\Delta_H &=(P\cap kH) \cap k\Delta_H\\
&= \left(\bigcap_{x\in G} Q^x\right) \cap k\Delta_H\\
&=\bigcap_{x\in G}\left(Q\cap k\Delta_H\right)^x\\
&=\bigcap_{x\in G} T^x.
\end{align*}
Now, the conjugation action of $G$ on $\Delta_H$ has kernel $\mathbf{C}_G(\Delta_H)$, which contains $\mathbf{C}_G(\Delta)$ by \cite[Lemma 1.3(ii)]{woods-struct-of-G}. But $\mathbf{C}_G(\Delta) = \bigcap \mathbf{C}_G(a)$, where the intersection runs over a set of topological generators $a$ for $\Delta$, and each $\mathbf{C}_G(a)$ is open in $G$ by definition of $\Delta$. Now, as $\Delta$ is topologically finally generated, we see that $\mathbf{C}_G(\Delta)$ and hence $\mathbf{C}_G(\Delta_H)$ are also open in $G$.

That is, the conjugation action of $G$ on $\Delta_H$ factors through the finite group $G/\mathbf{C}_G(\Delta_H)$, and hence the intersections above are finite, so that (by the primality of $T$), we have
$$ S\cap k\Delta_H \subseteq T^x$$
for some $x\in G$.

Now, by assumption, $Q$ is standard, so $T$ is almost faithful. This means that
$$S^\dagger\cap \Delta_H \subseteq (T^\dagger)^x$$
is a finite group, and so, since $[\Delta : \Delta_H] < \infty$, we have that $S^\dagger$ is also finite, so $S$ is almost faithful.

It remains to show that $S^\circ kG = P$. By Lemma \ref{lem: extension of G-prime from kDelta to kG is prime}, we see that $S^\circ kG = P'$ is a prime ideal of $kG$ contained in $P$. Now,
\begin{align*}
(P\cap kH)kG &= \left( \bigcap_{g\in G} Q^g\right) kG\\
& = \left( \bigcap_{g\in G} \left(\bigcap_{h\in H} T^h kH\right)^g \right) kG\\
&= \left( \bigcap_{g\in G} T^g\right) kG\\
&= (P\cap k\Delta_H) kG&\text{by calculation above}\\
&\subseteq (P\cap k\Delta)kG = P' \subseteq P,
\end{align*}
and as $H$ is open and normal in $G$, we know from Lemma \ref{lem: G-prime stuff}(i) that $P$ is a minimal prime above $(P\cap kH)kG$, so that $P = P'$.
\end{proof}

Finally, the main theorem of this subsection:

\begin{thm}\label{thm: P is induced from nabla}
Let $G$ be a nilpotent-by-finite compact $p$-adic analytic group, $P$ a prime ideal of $kG$, $H$ an orbitally sound open normal subgroup of $G$, $Q$ a minimal prime ideal above $P\cap kH$, and $N = \mathrm{i}_G(Q^\dagger)$. Then there exists an ideal $L \lhd k[[\nabla_G(N)]]$ with $P = L^G$.
\end{thm}

\begin{rk}
The subgroup $N$ is a \emph{vertex} of the prime ideal $P$, and the ideal $L$ is a \emph{source} of $P$ corresponding to the vertex $N$.
\end{rk}

\begin{proof}
We follow the proof of \cite[2.3]{lp}, as reproduced in \cite[20.5]{passmanICP}.

Trivially, $H$ stabilises $Q$, i.e. $H \leq B := \mathrm{Stab}_G(Q)$; and $B$ normalises $Q^\dagger$. Set $N := \mathrm{i}_G(Q^\dagger)$. Now we must have $\mathbf{N}_G(Q^\dagger) \leq A := \mathbf{N}_G(N)$: indeed, if $x\in G$ normalises $Q^\dagger$, then it permutes the (finitely many) closed orbital subgroups $K$ of $G$ containing $Q^\dagger$ as an open subgroup, and hence it normalises $N$, which is generated by those $K$ \cite[Definition 1.6]{woods-struct-of-G}.

We are in the following situation:

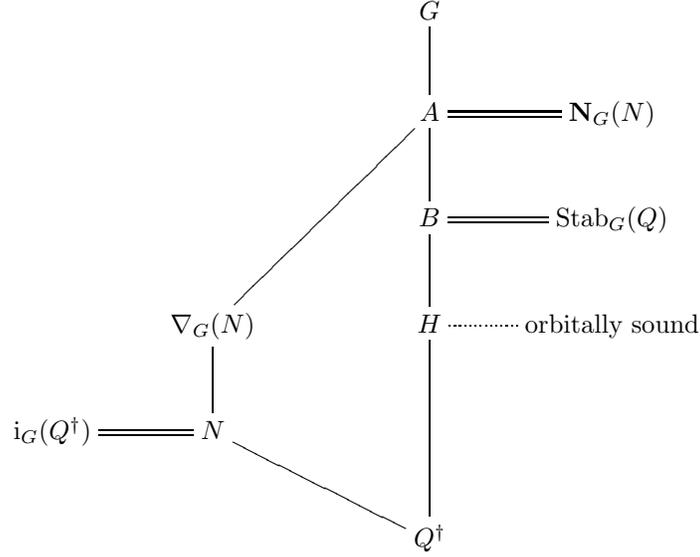
\centerline{
\xymatrix{
&&&G\ar@{-}[d]\\
&&&A\ar@{=}[r]\ar@{-}[d]\ar@{-}[ddll]&\mathbf{N}_G(N)\\
&&&B\ar@{=}[r]\ar@{-}[d]&\mathrm{Stab}_G(Q)\\
&\nabla_G(N)\ar@{-}[d]&&H\ar@{-}[dd]\ar@{..}[r]&\mbox{orbitally sound}\\
\mathrm{i}_G(Q^\dagger)\ar@{=}[r]&N\ar@{-}[drr]\\
&&&Q^\dagger
}
}

Now, Lemma \ref{lem: induced ideals from kA} shows that there is a prime ideal $T$ of $kA$ with $P = T^G$ and $T\cap kH  = \bigcap_{a\in A} Q^a$. It will suffice to show the existence of a prime ideal $L$ of $k[[\nabla_G(N)]]$ with $T =L^A$, by Lemma \ref{lem: induction of ideals}.

Let $M$ be an open characteristic subgroup of $N$ contained in $Q^\dagger$, whose existence is guaranteed by Lemma \ref{lem: open char subgroup}. Write $\nabla = \nabla_G(N)$, which we know is equal to $\nabla_A(M)$ by Lemma \ref{lem: sub N = mod M}, and denote by $\overline{(\cdot)}$ images under the natural map $kA \to k[[A/M]]$.

Now $Q$ is a prime ideal of $kH$ with $M\leq Q^\dagger$ an open subgroup, so $\overline{Q}$ is an almost faithful prime ideal of $k\overline{H}$; hence, as $\overline{H}$ is orbitally sound \cite[Lemma 1.5(ii)]{woods-struct-of-G}, we see that $\overline{Q}$ is a \emph{standard} prime of $k\overline{H}$.

But $T\cap kH  = \bigcap_{a\in A} Q^a$ clearly implies $\overline{T}\cap k\overline{H}  = \bigcap_{\overline{a}\in \overline{A}} \overline{Q}^{\overline{a}}$, by the modular law. Now Lemma \ref{lem: T standard iff Q standard} implies that $\overline{T}$ is also a standard prime ideal of $k\overline{A}$: that is, there is an almost faithful prime ideal $\overline{L}$ of $k[[\Delta(\overline{A})]]$ with $\overline{T} = \overline{L}^{\overline{A}}$. Lifting this back to $kA$, we see that we have an almost faithful mod $M$ prime ideal $L$ of $k\nabla$ with $T = L^A$ as required.
\end{proof}

%\textit{Proof of Theorem \hyperref[thm: theorem C]{C}(ii).} This follows from Theorem \ref{thm: P is induced from nabla}.\qed

We end this subsection with an important application of this theorem. Recall the definition of $\nio(G)$ from \cite[Definition 2.5]{woods-struct-of-G}.

\begin{cor}\label{cor: faithful primes are induced from a proper open subgroup}
Suppose $G$ is a nilpotent-by-finite compact $p$-adic analytic group which is not orbitally sound. Let $P$ be a faithful prime ideal of $kG$. Then $P$ is induced from some proper open subgroup of $G$ containing $\nio(G)$.
\end{cor}

\begin{proof}
Write $H = \nio(G)$. $H$ is orbitally sound by \cite[Theorem 2.6(ii)]{woods-struct-of-G}.

Let $Q$ be a minimal prime ideal above $P\cap kH$, so that $N = \mathrm{i}_G(Q^\dagger)$ is a vertex for $P$ by Theorem \ref{thm: P is induced from nabla}. Then $P$ is induced from $\nabla_G(N)$, which is contained in $\mathbf{N}_G(N)$, and so $P$ is induced from $\mathbf{N}_G(N)$ itself by Lemma \ref{lem: induction of ideals}. But, as $\nio(G)$ is orbitally sound, in particular it must normalise $N$ \cite[Theorem 2.6(i)]{woods-struct-of-G}. Hence, if $\mathbf{N}_G(N)$ is a proper subgroup of $G$, we are done.

Suppose instead that $\mathbf{N}_G(N) = G$, i.e. that $\mathrm{i}_G(Q^\dagger)$ is a normal subgroup of $G$. Then, for each $g\in G$, $(Q^\dagger)^g$ is a finite-index subgroup of $\mathrm{i}_G(Q^\dagger)$ \cite[Proposition 1.7]{woods-struct-of-G}; and $Q^\dagger$ is orbital in $G$, so there are only finitely many $(Q^\dagger)^g$, and their intersection $(Q^\dagger)^\circ$ must also have finite index in $\mathrm{i}_G(Q^\dagger)$. But $(Q^\dagger)^\circ = P^\dagger = 1$, so in particular we have $\mathrm{i}_G(Q^\dagger) = \Delta^+$, and hence $P$ is induced from $\nabla_G(N) = \Delta$, again by Theorem \ref{thm: P is induced from nabla}. Hence, as $\nio(G)$ contains $\Delta$, $P$ must be induced from $\nio(G)$ itself.
\end{proof}

\subsection{The general case: inducing from open subgroups}

Now we will proceed to show that $kG$ is catenary.

\begin{lem}\label{lem: ideals maximal subject to extension contained in P}
Let $H$ be an open subgroup of $G$, and $P$ a prime ideal of $kG$. Suppose $Q$ is an ideal of $kH$ maximal amongst those ideals $A$ of $kH$ with $A^G\subseteq P$. Then $Q$ is prime, and $P$ is a minimal prime ideal above $Q^G$.
\end{lem}

\begin{proof}
Suppose $I$ and $J$ are ideals strictly containing $Q$: then, by the maximality of $Q$, we see that $I^G$ and $J^G$ must strictly contain $P$. Hence $I^G J^G \subseteq (IJ)^G$ \cite[Lemma 14.5]{passmanICP} strictly contains $P$, and so $IJ$ strictly contains $Q$. Hence $Q$ is prime.

$P$ is clearly a prime ideal containing $Q^G$, so to show it is minimal it suffices to find any ideal $A$ of $kH$ with $P$ a minimal prime above $A^G$. Let $N$ be the normal core of $H$ in $G$, and take $A = (P\cap kN)^H$: then by Lemma \ref{lem: induction of ideals} we have $A^G = (P\cap kN)^G = (P\cap kN)kG$, and by Lemma \ref{lem: G-prime stuff}(i), $P$ is a minimal prime above this.
\end{proof}

\begin{lem}\label{lem: catenary subgroup}
Let $H$ be an open subgroup of $G$ with $kH$ catenary. If $P\lneq P'$ are adjacent primes of $kG$, and $P$ is induced from $kH$, then $h(P') = h(P) + 1$.
\end{lem}

\begin{proof}
(Adapted from \cite[3.3]{ll}.)
Choose an ideal $Q$ (resp. $Q'$) of $kH$ which is maximal amongst those ideals $A$ of $kH$ with $A^G\subseteq P$ (resp. $A^G \subseteq P'$). Then $Q$ and $Q'$ are prime, and $P$ (resp. $P'$) is a minimal prime ideal over $Q^G$ (resp. $Q'^G$), by Lemma \ref{lem: ideals maximal subject to extension contained in P}. Hence, by Proposition \ref{propn: P minimal over Q implies heights are equal}, we see that it suffices to show that $h(Q') = h(Q) + 1$.

Suppose not. Then there exists some prime ideal $I$ of $kH$ with $Q\lneq I\lneq Q'$; and we may choose a prime ideal $J$ of $kG$ which is minimal over $I^G$. Then $P\leq J\leq P'$. But $h(Q) < h(I) < h(Q')$ implies (by another application of Proposition \ref{propn: P minimal over Q implies heights are equal}) that $h(P) < h(J) < h(P')$, contradicting our assumption that $P$ and $P'$ were adjacent primes.
%
%Write $P = Q^H$ for some ideal $Q$ of $kH$, which we may assume to be prime by taking it to be maximal. $R = kH \subset S = kG$ satisfies the hypotheses of \cite[proposition 2.2]{ll}, by Lemma \ref{lem: sslc for kG}. We may then apply \cite[proposition 2.3]{ll} (mutatis mutandis) to show that $h(P) = h(Q)$. Similarly, we may choose $Q'$ containing $Q$ maximal amongst $kH$-ideals with $(Q')^G = P'$, and we see that $h(P') = h(Q')$. Now we only need to show that $Q$ and $Q'$ are adjacent, as $kH$ is catenary. But this is clear: $Q \lneq I \lneq Q'$ would imply $P = Q^G \lneq I^G \lneq (Q')^G = P'$, and $P'$ would contain a prime $J$ of $kG$ minimal over $I^G$ with $h(J) \neq h(P)$. \cite[3.2]{ll}
\end{proof}

\begin{cor}\label{cor: main result}
Let $G$ be a nilpotent-by-finite compact $p$-adic analytic group, and $k$ a finite field of characteristic $p > 2$. Then $kG$ is a catenary ring.
\end{cor}

\begin{proof}
(Adapted from \cite[3.3]{ll}.)
Take two adjacent prime ideals $P \lneq Q$ of $kG$, and assume without loss of generality that $P$ is faithful. We proceed by induction on the index $[G: \mathrm{nio}(G)]$. When this index equals $1$, we are already done by Theorem \ref{thm: catenarity in o.s. case}, so suppose not. Then Corollary \ref{cor: faithful primes are induced from a proper open subgroup} implies that $P$ is induced from some proper open subgroup $H$ of $G$ containing $\nio(G)$. As $\nio(G)$ is an orbitally sound open normal subgroup of $H$, it must be contained in $\nio(H)$ (by the maximality of $\nio(H)$), and so we have $[H:\mathrm{nio}(H)] < [G:\mathrm{nio}(G)]$. By induction, $kH$ is catenary, so we may now invoke Lemma \ref{lem: catenary subgroup} to show that $h(Q) = h(P) + 1$.
\end{proof}

%\textit{Proof of Theorem \hyperref[thm: theorem A]{A}.} This follows from Corollary \ref{cor: main result}.\qed

\bibliography{biblio}
\bibliographystyle{plain}

\end{document}